\documentclass{amsart}

\usepackage{amssymb, amsmath, enumerate, amsfonts, amsthm, mathrsfs, url, bm, mathtools,tikz, color, xcolor,subcaption,comment}
\usepackage[backref]{hyperref}
\hypersetup{
	colorlinks,
    linkcolor={blue!60!black},
    citecolor={blue!60!black},
    urlcolor={blue!60!black}
}
\usetikzlibrary{backgrounds,calc}
\usetikzlibrary { decorations.pathmorphing, decorations.pathreplacing, decorations.shapes, }

\usetikzlibrary{graphs,graphs.standard}
\pgfdeclarelayer{bg}
\pgfdeclarelayer{fg}
\pgfsetlayers{bg,main,fg}

\numberwithin{equation}{section}

\newcommand{\ls}{\lesssim}
\newcommand{\bbq}{\mathbb{Q}}
\newcommand{\bbz}{\mathbb{Z}}
\newcommand{\gs}{\gtrsim}

\newcommand{\abs}[1]{\mathopen{}\left| #1\mathclose{}\right|}

\newcommand{\brac}[1]{\mathopen{}\left( #1 \mathclose{}\right)}
\newcommand{\norm}[1]{\mathopen{}\left\| #1\mathclose{}\right\|}

\newcommand{\bbr}{\mathbb{R}}
\newcommand{\bbc}{\mathbb{C}}

\newcommand{\Abs}[1]{\lvert #1\rvert}

\newcommand{\bbf}{\mathbb{F}}
\newtheorem{theorem}{Theorem}[section]

\newtheorem{corollary}[theorem]{Corollary}

\newtheorem{lemma}[theorem]{Lemma}

\title{The sum-product conjecture is false for real numbers}
\author{Thomas F. Bloom}
\address{Department of Mathematics, University of Manchester, Manchester, M13 9PL}
\email{thomas.bloom@manchester.ac.uk}
\author{Will Sawin}
\address{Department of Mathematics, Princeton University, Princeton, NJ 08540}
\email{wsawin@math.princeton.edu}
\author{Carl Schildkraut}
\address{Department of Mathematics, Stanford University, Stanford CA}
\email{carlsch@stanford.edu}
\author{Dmitrii Zhelezov}

\begin{document}
\begin{abstract}
We disprove the sum-product conjecture for real numbers by constructing arbitrarily large $A\subset \bbr$  (whose elements are algebraic integers in a number field of degree $\asymp \log\abs{A}$) such that
\[\max(\abs{A+A},\abs{AA})\leq \abs{A}^{2-c}\]
where $c>0$ is an absolute constant.

We also disprove the many sums and products conjecture by constructing, for any $k\geq 3$, arbitrarily large $A\subset \bbr$ such that
\[\max(\abs{kA},\Abs{A^{(k)}})\leq \abs{A}^{C\frac{\log k}{\log\log k}}\]
for some constant $C>0$. We obtain similar constructions for $p$-adics, finite fields, and function fields in positive characteristic, and also obtain new lower bounds for the number of solutions to linear equations in a multiplicative group, and the number of solutions to the unit equation in sufficiently many variables.
\end{abstract}
\maketitle

\section{Introduction}
Given any finite set $A$ in some ring we define the sum set and product set of~$A$~as 
\[A+A=\{a+b: a,b\in A\}\quad\textrm{and}\quad AA=\{ ab : a,b\in A\}.\]
The sum-product conjecture in a given ring is that at least one of these must grow near-maximally; more precisely
\begin{equation}\label{eq-sp}\max(\abs{A+A},\abs{AA})\geq \abs{A}^{2-o(1)}\end{equation}
(where the $o(1)$ term tends $0$ as $\abs{A}\to \infty$). This is often attributed to Erd\H{o}s and Szemer\'{e}di, who proved the first results in this direction \cite{ErSz83}, but it first appeared in the literature in a paper of Erd\H{o}s in 1976 \cite{Er76} (in which he says he first made the conjecture 18 months earlier). This question makes sense over any ring (although there are obvious complications in finite rings or rings with zero divisors). Erd\H{o}s \cite{Er76} asked this specifically for $\bbz$, $\bbr$, and $\bbc$, although his main interest was for $A\subset \bbz$.

Most proofs in the sum-product literature are geometric and combinatorial, using no number theory, and thus apply to any finite $A\subset \bbr$; the best result achieved in this direction so far is
\[\max(\abs{A+A},\abs{AA})\geq \abs{A}^{\frac{4}{3}+c-o(1)}\quad\textrm{for}\quad A\subset \bbr\]
for some small constant $c>0$. This was proved with $c=0$ by Solymosi \cite{So09} and for some $c>0$ by Konyagin and Shkredov \cite{KoSh16}. The value of $c$ has been improved a number of times since, with the current record of $c=\frac{10}{4407}$ due to Cushman \cite{Cu25}. 

In this paper we prove that the sum-product conjecture \eqref{eq-sp} is false over the reals, by constructing arbitrarily large counterexamples in totally real algebraic number fields of large degree. The degree of these fields tends to infinity as the sets grow (like $\asymp \log n$ for a counterexample of size $n$), and so \eqref{eq-sp} may still be true in number fields of bounded degree (and, in particular, the original setting of $\bbz$).

\begin{theorem}\label{th-headline}
There exists an absolute constant $c>0$ such that there are arbitrarily large finite $A\subset\bbr$ with
\[\max(\abs{A+A},\abs{AA})\leq \abs{A}^{2-c}.\]
\end{theorem}

Our arguments deliver an explicit value of $c$, but this is a tedious calculation, and the value obtained is very small. Since the main interest of this result is the existence of an absolute constant $c>0$, we have chosen to present a non-explicit version of the proof, to better demonstrate the main ideas. In Section~\ref{sec-numerical} we sketch how a version of the proof with explicit constants can deliver $c\geq 0.00000087$ (although this should not be taken too seriously, and can certainly be improved with a little more effort).

Theorem~\ref{th-headline} is easily deduced from the following more general result.
\begin{theorem}\label{th-main}
There exists an absolute constant $C>0$ such that the following holds. There are infinitely many $d$, with accompanying totally real number fields $K$ of degree $d$ over $\bbq$, such that, for any $X\geq 1$, there exists $A\subset \mathcal{O}_K$ with 
\[X^{d} \leq \abs{A}\leq (CX)^{d},\]
\[\abs{A+A}\leq C^{d}\abs{A}\textrm{, and}\quad\abs{AA}\leq 2^{-d}\abs{A}^2.\]
\end{theorem}
By choosing an arbitrary embedding of $K$ into $\bbr$, and $X=C^{1/\epsilon}$, we deduce the following, which provides examples in which the sum set is very small, and yet there is still a power saving on the size of the product set.
\begin{corollary}\label{cor-main}
There exists an absolute constant $c>0$ such that the following holds. For any $\epsilon\in (0,1)$ there are arbitrarily large $A\subset \bbr$ with
\[\abs{A+A}\leq \abs{A}^{1+\epsilon}\textrm{ and }\abs{AA}\leq \abs{A}^{2-c\epsilon}.\]
\end{corollary}

Theorem~\ref{th-headline} is an immediate consequence. This result should also be compared to the lower bound of Solymosi \cite{So09}, who proved that for any $A\subset \bbr$
\begin{equation}\label{eq-sol}
\abs{A+A}^2\abs{AA}\geq \abs{A}^{4-o(1)}.
\end{equation}

Using similar ideas we also obtain new lower bounds in a number of other related problems of a sum-product flavour, which we summarise below.

\subsection{Many sums and products}
Erd\H{o}s \cite{Er77} also made the stronger conjecture that, for any $k\geq 2$ and $\epsilon>0$,
\[\max(\abs{kA},\Abs{A^{(k)}}) \gg_{k,\epsilon}\abs{A}^{k-\epsilon},\]
where $kA$ and $A^{(k)}$ denote the $k$-fold sum set and $k$-fold product set, respectively. (Once again this was for $A\subset \bbz$ originally, but Erd\H{o}s and Szemer\'{e}di \cite{ErSz83} asked it also for $A\subset \bbr$). Our methods provide a strong counterexample to this conjecture for large $k$.
\begin{theorem}\label{th-mainmult} 
There exists an absolute constant $C>0$ such that for any fixed $k\geq 3$ there exist arbitrarily large $A\subset \bbr$ with
\[\max(\abs{kA},\Abs{A^{(k)}})\leq \abs{A}^{C\frac{\log k}{\log\log k}}.\]
Furthermore, for any fixed $\epsilon\in (0,1)$, there exist arbitrarily large $A\subset \bbr$ such that
\[\max(\abs{kA},\Abs{A^{(k)}})\leq \abs{A}^{C^{1/\epsilon}+\epsilon \log k}\textrm{ for all }k\geq 3.\]
\end{theorem}
This is likely the best possible dependence on $k$ in the exponent (at least for $k$ fixed and $\abs{A}\to \infty$). It follows from work of Mudgal \cite{Mu24} and the recent resolution of the weak polynomial Freiman--Ruzsa conjecture by Gowers, Green, Manners, and Tao \cite{GGMT25} that, for any $A\subset \bbr$, for all $k\geq 3$,
\[\max(\abs{kA},\Abs{A^{(k)}})\geq\abs{A}^{(\log k)^c}\]
for some absolute constant $c>0$. With sufficiently improved bounds on the number of solutions to linear equations in multiplicative groups (as discussed in Section~\ref{sec-linear}) this exponent can likely be improved to $\frac{\log k}{\log\log k}$. This exact quantitative dependence was achieved for $A\subset \bbz$ by P\'{a}lv\"{o}lgyi and Zhelezov \cite{PaZh21}. In a similar vein, Konyagin \cite{Ko14} proved that if $A\subset \bbc$ is a finite set such that $\abs{AA}\leq \abs{A}^{1+O(1/k)}$ then $\abs{kA}\geq \abs{A}^{c\log k}$ for some constant $c>0$. 

\subsection{Linear equations in a multiplicative group}
Another application of our construction is to provide new lower bounds for the number of solutions to linear equations in a multiplicative group.
\begin{theorem}\label{th-linearmain}
There is an absolute constant $C>1$ such that the following holds. Let $k\geq C$ be any integer. There exist infinitely many $d\geq 2$ and multiplicative groups $\Gamma\leq\bbr^\times$ of rank $\leq d$ such that there are at least
\[\geq (Ck)^{kd}\]
solutions to
\[x_1+\cdots+x_k=1\]
with $x_i\in \Gamma$ and $x_i>0$ for $1\leq i\leq k$.
\end{theorem}
In particular this shows that the dependence on $d$ in the corresponding upper bound of Evertse, Schlickewei, and Schmidt \cite{ESS02} is the best possible (for subgroups of $\bbr^\times$). Similarly, we produce a new lower bound for the number of solutions to the unit equation $x_1+\cdots+x_k=1$ with $x_i\in \mathcal{O}_K^\times$, provided $k$ is sufficiently large (in absolute terms).
\begin{theorem}\label{th-unitlower}
There exists an integer $k\geq 2$ and an absolute constant $C>1$ such that, for infinitely many $d$, there exists a number field $K$ of degree $d$ such that the equation
\[x_1+\cdots+x_k=1\]
has at least $C^d$ many solutions with $x_i\in \mathcal{O}_K^\times$.
\end{theorem}
By contrast, when $k=2$ this equation is conjectured to have sub-exponential in $d$ many solutions. For more details, and further discussion of related results, see Section~\ref{sec-linear}.

\subsection{Sum-product in other settings}
Finally, in Section~\ref{sec-variants} we discuss variants of our construction. We first obtain, via only a slight modification of the argument, analogues of all of the above results for $A\subset \bbq_p$ for any prime $p$. 

The second variant, obtained by taking the construction in $\bbq_p$ `modulo $p$' in a suitable sense, provides an upper bound for the sum-product problem in the finite field $\bbf_p$, another natural setting. We state a slightly simplified version here (see Theorem~\ref{finite-field-main} for the full version).
\begin{theorem}
There exists a constant $c>0$ such that, for all sufficiently large primes $p$, there exists $A \subset \mathbb F_p$ with $p^{c}< \abs{A} < p^{1/2}$ and $\max(\abs{A+A},\abs{AA})\leq \abs{A}^{2-c}$. \end{theorem}
We note that such a result for $\abs{A}$ very small in terms of $p$ is a consequence of Theorem~\ref{th-headline} combined with the transference method of Vu, Wood, and Wood \cite{VWW11}, but here we are able to find sum-product counterexamples which are reasonably `large' (in that $\abs{A}\geq p^c$ for some constant $c>0$).

The first sum-product results for subsets of $\bbf_p$ of size $p^\delta$ for some $\delta>0$ were obtained by Bourgain, Katz, and Tao \cite{BKT04}. The best result thus far obtained in this direction is due to Mohammadi and Stevens \cite{MoSt23}, who proved that if $\abs{A}< p^{1/2}$ then 
\[\max(\abs{A+A},\abs{AA})\geq \abs{A}^{\frac{5}{4}-o(1)}.\]

The third variant, which is more involved, constructs sum-product counterexamples in infinite fields of fixed positive characteristic.

\begin{theorem}\label{th-ffmain}
There exists an absolute constant $c>0$ such that, for any prime $p$, if $q$ is an power of $p$ then there exist arbitrarily large $A\subset \bbf_{q}((t))$ such that
\[\max(\abs{A+A},\abs{AA})\leq \abs{A}^{2-\frac{c}{\log p}}.\]
\end{theorem}
In small characteristics the exponent saving is reasonably good -- for example when $q=1024$ we obtain 
\[\max(\abs{A+A},\abs{AA})\leq \abs{A}^{1.906}.\]
The best known lower bound for the sum-product problem in function fields, due to Bloom and Jones \cite{BlJo14}, is that for any $q$ and $A\subset \bbf_q((t))$, 
\[\max(\abs{A+A}, \abs{AA})\geq  q^{-1/5}\abs{A}^{6/5-o(1)}\]
(where the $o(1)$ exponent tends to zero as $\abs{A}\to \infty$).
\subsection*{The role of AI in this proof} 
The authors were inspired to revisit
the possibility of disproving the sum-product conjecture using number
fields of large degree by the recent OpenAI counterexample to the unit distance conjecture (see \cite{ABGLSSTWW2026UnitDistance}). Curiously, the final construction given here required far less number theoretic input than the unit distance counterexample. GPT-5.5 Pro was used as a sounding board in the early stages of the development of this proof, but the final proof, including all the main ideas, was almost entirely human-generated (the exception being the suggestion of Lemma~\ref{lem-unitsep}, which replaced a more complicated result of Schinzel with a short elementary argument). Everything in this paper was written by the authors.

\subsection*{Acknowledgements}

We thank Akshat Mudgal for clarifying the quantitative aspects of \cite{Mu24}, and suggesting that our construction could also be used to prove something like Theorem~\ref{th-linearmain}. We thank Jacob Fox and Sarah Peluse for many helpful comments and Spencer Dembner for his careful reading of an earlier version of this article. 

TB is a Royal Society University Research Fellow. 
WS is supported by NSF grant DMS-2502029 and is a Sloan Research Fellow. 
CS is supported by the National Science Foundation Graduate Research Fellowship Program under Grant No.~DGE-2146755.

\section{Sketch of the construction}
In this section we give a sketch of the construction which is used to prove Theorem~\ref{th-main}. The construction is a high-dimensional version of the standard Balog--Wooley 
example first introduced in \cite{BalogWooleyLowEnergy}.  In its simplest one-dimensional form, one takes
\[
        A=GP,
\]
where \(G\) is a short geometric progression and \(P\) is an interval.  The
multiplicative structure of \(G\) gives\footnote{In this section we will make free use of the Vinogradov notation $\ll$ and $\gg$, both of which indicate the relevant inequality holds up to some absolute constant.}
\[
        \abs{GG}\ll \abs{G},
\]
while the additive structure of \(P\) keeps \(GP+GP\) inside a relatively short
interval.  This gives examples for which both \(\abs{A+A}\) and \(\abs{AA}\) are
smaller than the trivial bound \(\abs{A}^{2}\), but only by a logarithmic factor
in \(\abs{A}\).  This example was used by Balog and Wooley \cite{BalogWooleyLowEnergy} to show that the natural additive energy variant of the sum-product conjecture is false, yet it falls short of being a counterexample to the original conjecture since the geometric progression is exponentially sparse, which allows $A+A$ to still have size $\geq \abs{A}^{2-o(1)}$.

The point of the present construction is to transform this example, replacing the geometric progression with a much denser multiplicatively structured set, and the arithmetic progression with a high-dimensional lattice embedded in \(\mathbb R\).  Instead of working in \(\mathbb Z\), we work in the ring of integers
\(\mathcal O_K\) of a totally real number field \(K\) of large degree \(d\).  The
\(d\) real embeddings
\[
        \sigma_1,\ldots,\sigma_d\colon K\hookrightarrow \mathbb R
\]
allow us to view \(\mathcal O_K\) as a lattice of full rank in \(\mathbb R^d\).  We will
use two different lattice structures in \(\mathcal O_K\): the additive lattice
of algebraic integers and the multiplicative logarithmic lattice of units.

The additive part of the construction is a box of algebraic integers.  We choose
a large parameter \(X\), and take \(P\subset \mathcal O_K\) so that every
embedding of every \(p\in P\) lies in a short interval around \(X\), say
\[
        \sigma_i(p)\in [X-cX,X+cX]\qquad(1\leq i\leq d)
\]
for some small constant $c>0$. The geometry of numbers gives
\[
        \abs{P}\gg X^d\Delta_K^{-1/2},
\]
up to harmless constants. Thus, provided the discriminant $\Delta_K$ is bounded above by $O(1)^d$, \(P\) behaves like a \(d\)-dimensional additive box of size $\gg X^d$.

The multiplicative part is a box in the unit lattice.  By Dirichlet's unit
theorem, the logarithms of the absolute values of the embeddings of units form a
lattice of rank \(d-1\) in the hyperplane defined by the equation $x_1+\cdots+x_d=0$. We choose
\[
        G=\{u\in\mathcal O_K^\times:
        \abs{\log\abs{\sigma_i(u)}}\leq Y\ \text{for all }i\}.
\]
The regulator of the number field controls the covolume of this unit lattice, and hence, provided the regulator is at most $O(1)^d$, there are $\gg Y^{d-1}$ many such units.  Moreover, since \(GG\) is contained in the same logarithmic lattice
box with \(Y\) replaced by \(2Y\), we have a small-doubling estimate of the form
\[
        \abs{GG}\leq O(1)^d\abs{G}.
\]
This is the high-dimensional analogue of the fact that a geometric progression
has small product set.

We then take
\[
        A=GP=\{up:u\in G,\ p\in P\}.
\]
Importantly, provided $X$ and $Y$ are chosen suitably, this product is direct in the sense that $\abs{A}=\abs{G}\abs{P}$. This is because, if $u_1p_1=u_2p_2$ with $u_i\in G$ and $p_i\in P$, then $p_1/p_2$ is a unit. Provided the short interval around $X$ in the definition of $P$ is sufficiently short, we have $\sigma_i(p_1/p_2)\in [1-\epsilon,1+\epsilon]$ for all $1\leq i\leq d$ and any small absolute constant $\epsilon>0$. By a result of Schinzel, however, if $u\neq1$ is a unit then there exists $1\leq i\leq d$ such that $\Abs{\sigma_i(u)-1}>\epsilon$, where $\epsilon>0$ is an absolute constant independent of $K$ and $d$. Therefore the only solutions to $u_1p_1=u_2p_2$ are those with $u_1=u_2$. 

To control the size of $A+A$, the key point is that multiplication by units in \(G\) expands each embedding by
at most \(e^Y\).  Therefore every element of \(A\) lies, in all real embeddings,
inside a box of side length \(O(Xe^Y)\).  Consequently
\[
        A+A\subseteq
        \{\alpha\in\mathcal O_K:
        \abs{\sigma_i(\alpha)}\ll Xe^Y\ \text{for all }i\}.
\]
The additive lattice-counting estimate then gives
\[
        \abs{A+A}\leq O(e^YX)^d\leq O(e^Y)^d\abs{A},
\]
since \(\abs{P}\asymp X^d\) and \(\abs{A}=\abs{G}\abs{P}\).

On the product side we use
\[
        AA\subseteq GGP P.
\]
The set \(GG\) has size only \(O(1)^d\abs{G}\), while trivially $\abs{PP}\leq \abs{P}^2$, and so
\[
        \abs{AA}
        \leq O(1)^d \abs{G}\abs{P}^{2}
        \leq O(1/Y)^{d-1}\abs{A}^2.
\]
The saving in the product set is therefore roughly the size of the unit box
\(G\), which is $\geq (cY)^{d-1}$. Since the saving in $\abs{AA}$ is $O(1/Y)^{d-1}$, if $Y$ is chosen as a sufficiently large absolute constant then
\[
        \abs{AA}\leq 2^{-d}\abs{A}^{2}.
\]
On the other hand, once $Y$ is fixed, \(X\) may be chosen large enough, depending on a prescribed \(\epsilon>0\), so
that the factor \(O(e^Y)^d\) in the sumset estimate is bounded above by 
\(\abs{A}^{\epsilon}\).  This gives
\[
        \abs{A+A}\leq \abs{A}^{1+\epsilon}\quad\textrm{and}\quad
        \abs{AA}\leq \abs{A}^{2-c\epsilon}
\]
for some constant $c>0$. 

All that remains is to show that we can perform the above construction for arbitrarily large $A$, which means (since $X$ and $Y$ are constants, so $\abs{A}$ grows like $O(1)^d$) that we need $d\to \infty$. In other words, we need a supply of number fields with degree \(d\to\infty\), in which both the 
discriminant and the regulator (which control the covolume of the additive lattice and multiplicative lattice respectively) grow at most exponentially in \(d\).  Such
bounded-root-discriminant towers go back to Martinet's use of class field
towers, and the regulator control follows from standard Brauer--Siegel type
bounds in this setting. The regulator control has already found applications outside number theory, being used to construct explicit lattice sphere packings~\cite[\S4]{TS91}.

\section{Algebraic number theory}
In this section we review the necessary concepts required from algebraic number theory; with the exception of Theorem~\ref{th-tower}, these are all classical results that can be found in most textbooks on the subject.

Let $K$ be a totally real number field of degree $d$ over $\bbq$, and let $\Delta_K$ be the discriminant of $K$ (which is strictly positive if $K$ is totally real). Let $R_K$ be the regulator of $K$. For those unfamiliar with algebraic number theory, the important role of these parameters for our purposes is that they control the covolume of the lattices of the algebraic integers and units respectively. The only fact that we will require about $K$ (aside from it being totally real) is that these are both bounded above by $O(1)^d$. This is usually done with emphasis on $\Delta_K$, but similar control on $R_K$ follows from the following lemma.
\begin{lemma}\label{lem-regbound}
If $K$ is a totally real number field of degree $d\geq 2$ then
\[R_K \leq \Delta_K.\]
\end{lemma}
\begin{proof}
As in the proof of \cite[XIII, Theorem 3]{La70}, for any real $s>1$, if $h_K$ is the class number of $K$ then
\[2^dR_Kh_K\leq 2s(s-1)(\pi^{-d/2}\Delta_K^{1/2})^s \Gamma(s/2)^d\zeta(s)^d.\]
In particular, letting $s=2$, since $h_K\geq 1$ and $d\geq 2$,
\[R_K\leq R_Kh_K \leq 4(\pi/12)^d\Delta_K\leq \Delta_K.\qedhere\]
\end{proof}

It therefore suffices to produce $K$ with $\Delta_K\leq O(1)^d$, for arbitrarily large $d$. Such towers were first constructed by Martinet \cite{Ma78}.

\begin{theorem}[Martinet]\label{th-tower}
There exists an absolute constant $C>0$ such that, for infinitely many $d$, there exist totally real number fields $K$ with degree $d$ with $\Delta_K\leq C^d$.
\end{theorem}

There are $d$ embeddings $K\hookrightarrow \bbr$. These let us view the algebraic integers as $d$-dimensional lattices. As we are concerned with both sums and products, both the additive and multiplicative versions of these lattices will be useful to us.
\subsection{The additive lattice}
The ring of algebraic integers $\mathcal{O}_K$ can be viewed as a lattice of rank $d$ in $\bbr^d$ via the Minkowski embedding
\[\alpha \mapsto (\sigma_1(\alpha),\ldots,\sigma_d(\alpha)),\]
where $\sigma_1,\ldots,\sigma_d$ are the embeddings $\sigma_i\colon K\hookrightarrow \bbr$. The covolume of this lattice is $\Delta_K^{1/2}$ (see \cite[V, Lemma 2]{La70}). We write 
\[B^+(X) = \{ \alpha \in \mathcal{O}_K : \abs{\sigma_i(\alpha)}\leq X\textrm{ for all }1\leq i\leq d\}.\]
\begin{lemma}\label{lem-ball}
Let $K$ be a totally real number field of degree $d$. For any $X\geq 1$
\[X^{d}\Delta_K^{-1/2}\leq \Abs{B^+(X)}\leq (2X+1)^{d}.\] 
\end{lemma}
\begin{proof}
In the embedding described above, $B^+(X)$ is contained inside the $L^\infty$ ball of radius $X$. Moreover, points in this lattice are at least $1$-separated in the $L^\infty$ norm: if $x\neq y\in \mathcal{O}_K$ then, since $x-y$ is a non-zero algebraic integer, it has a non-zero integral norm. Furthermore, since $N(\alpha)=\prod_{i=1}^d \sigma_i(\alpha)$, we deduce
\[1\leq \Abs{N(x-y)}\leq \prod_{i=1}^d\Abs{\sigma_i(x-y)},\]
so there must exist $1\leq i\leq d$ such that $\Abs{\sigma_i(x)-\sigma_i(y)}\geq 1$. By a standard packing argument (for example, placing disjoint balls of radius $1/2$ around each lattice point) there are at most $(2X+1)^{d}$ many $1$-separated points in a ball of radius $X$, and we are done.

For the lower bound we use Blichfeldt's lemma: the covolume of the lattice is $\Delta_K^{1/2}$, and hence there exists some $a$ such that the number of lattice points in $a+\{ x\in \bbr^d : \norm{x}_\infty \leq X/2\}$ is at least
\[\frac{\mathrm{vol}(\{ x: \norm{x}_\infty \leq X/2\})}{\Delta_K^{1/2}}= \frac{X^{d}}{\Delta_K^{1/2}}.\]
The conclusion now follows by taking the difference set of these points.
\end{proof}

\subsection{The unit lattice}
The group of units $\mathcal{O}_K^\times$ of $K$ is the set of algebraic integers $\alpha$ such that $\alpha^{-1}$ is also an algebraic integer. By Dirichlet's unit theorem (see, for example, \cite[Chapter 1.7]{Ne99}) the group of units (modulo the roots of unity) $\mathcal{O}_K^\times/\{\pm 1\}$ can be viewed as a lattice of rank $d-1$ in $\bbr^{d}$ via the embedding
\[u\mapsto (\log \abs{\sigma_1(u)},\ldots,\log \abs{\sigma_d(u)}).\]
This is a lattice of rank $d-1$ inside the hyperplane
\[H=\{ x\in \bbr^d : x_1+\cdots+x_d=0\}.\]
The covolume of this lattice in $H$ is $\sqrt{d}R_K$, where $R_K$ is the regulator (see \cite[Chapter 1, Proposition 7.5]{Ne99}). We write 
\[B^\times(Y) = \{ \alpha \in \mathcal{O}_K^\times : \abs{\log\abs{\sigma_i(\alpha)}}\leq Y\textrm{ for all }1\leq i\leq d\}.\]
We note here the trivial, but crucial, fact that integers in $B^\times(Y)$ are still bounded in the additive sense also, so that
\[B^\times(Y)\subseteq B^+(e^Y).\]

It is a well-known fact that points in the unit lattice are separated by an absolute constant (independent of both the field and degree). This follows, for example, from  Schinzel's lower bound for the Mahler measure~\cite[Theorem 2]{Schinzel1973}. For our purposes the following simple lemma (suggested by GPT-5.5 Pro) will suffice. Let $\phi=\frac{1+\sqrt{5}}{2}$, so that $0\leq x^2+x^{-2}-2<1$ whenever $x\in (\phi^{-1},\phi)$.
\begin{lemma}\label{lem-unitsep}
If $u\in \mathcal{O}_K^\times$ and $\phi^{-1}< \Abs{\sigma_i(u)}<\phi$ for all $1\leq i\leq d$ then $u\in\{\pm 1\}$. 
\end{lemma}
\begin{proof}
Let $\alpha=u^2+u^{-2}-2\in \mathcal{O}_K$. If $\alpha \neq 0$ then, for each embedding $\sigma$, we have
\[0<\sigma(u)^2+\sigma(u)^{-2}-2<1,\]
so $\sigma(\alpha)\in (0,1)$. This contradicts that $N(\alpha)=\prod \sigma(\alpha)$ must be an integer. It follows that $\alpha=0$, whence $u^2=1$ and so $u\in \{\pm 1\}$. 
\end{proof}

The proof of the following is similar to that of Lemma~\ref{lem-ball}.
\begin{lemma}\label{lem-ball-mult}
Let $K$ be a totally real number field of degree $d$. For any $Y\geq 1$
\[Y^{d-1}d^{-1/2}R_K^{-1}\leq \Abs{B^\times(Y)}\leq 10(5Y+1)^{d-1}.\] 
\end{lemma}

In the proof of Lemma~\ref{lem-ball-mult} we will require bounds on the $(d-1)$-dimensional volume of $H\cap \{ x\in \bbr^d : \norm{x}_\infty \leq r\}$. Hensley \cite{He79} proved that 
\begin{equation}\label{eq-hensley}
    (2r)^{d-1}\leq \mathrm{vol}_{d-1}(H\cap \{ x\in \bbr^d : \norm{x}_\infty \leq r\})\leq 5(2r)^{d-1}.
\end{equation}
In fact, the central limit theorem implies that this volume is $\sim \sqrt{6/\pi}(2r)^{d-1}$ as $d\to \infty$ (a remark which Hensley attributes to Selberg).

\begin{proof}
Losing only a factor of $2$ (since $u$ and $-u$ are both mapped to the same vector) the set $B^\times(Y)$ can be viewed as a subset of the $L^\infty$ ball of radius $Y$, intersected with the hyperplane $H$. Moreover, by Lemma~\ref{lem-unitsep}, points in this lattice are at least $(\log \phi)$-separated in the $L^\infty$ norm. Indeed, if $x,y\in \mathcal{O}_K^\times$ and $x\not\in \{y,-y\}$ then $x/y\in \mathcal{O}_K^\times \backslash \{\pm 1\}$, and hence there exists $\sigma$ such that
\[ \Abs{\sigma(x)}/\Abs{\sigma(y)} \not\in(\phi^{-1},\phi).\]
Combining the same standard packing argument as in the proof of Lemma~\ref{lem-ball} with \eqref{eq-hensley}, there are at most $5(\frac{2}{c}Y+1)^{d-1}$ many $c$-separated points in $H$ intersected with an $L^\infty$ ball of radius $Y$. This proves the upper bound since $2/\log(\phi)<5$.

For the lower bound, we use the same idea as before: the covolume of the lattice is $\sqrt{d}R_K$, and hence there exists some $a$ such that the number of lattice points in $a+\{ x\in H : \norm{x}_\infty \leq Y/2\}$ is at least 
\[\frac{\mathrm{vol}_{d-1}(\{ x\in H: \norm{x}_\infty \leq Y/2\})}{\sqrt{d}R_K}\geq \frac{Y^{d-1}}{\sqrt{d}R_K}.\]
The conclusion now follows by taking the difference set of these points. 
\end{proof}

\section{The construction}
In this section we use the algebraic number theory facts of the previous section to prove Theorems~\ref{th-main} and \ref{th-mainmult}.

\begin{lemma}\label{lem-main}
There exists an absolute constant $c>0$ such that the following holds. Let $K$ be a totally real number field of degree $d\geq 2$ with discriminant $\Delta_K$ and let $X,Y\geq 2$. There exists a set $A\subset \mathcal{O}_K$ such that
\[\frac{(cXY)^d}{Y\Delta_K^{3/2}}\leq \abs{A}\leq (XY/c)^d,\]
\[\Abs{AA}\leq c^{-d}Y^{1-d}\Delta_K^2\abs{A}^2,\]
and
\[\abs{A+A}\leq (e^Y/c)^d\Delta_K^{1/2}\abs{A}.\]
\end{lemma}

Theorem~\ref{th-main} is an immediate consequence, letting $K$ be a totally real field of sufficiently large degree $d\geq 2$ with $\Delta_K \leq C^d$ for some absolute constant $C>0$, as provided by Theorem~\ref{th-tower}, and choosing $Y=4C^2c^{-1}$, say.

\begin{proof}
Without loss of generality, we can assume that $X$ and $Y$ are both sufficiently large (in absolute terms), and that $X$ is an integer. Let $G=B^\times(Y)$, so that by Lemma~\ref{lem-ball-mult} and Lemma~\ref{lem-regbound}
\[Y^{d-1}d^{-1/2}\Delta_K^{-1}\leq Y^{d-1}d^{-1/2}R_K^{-1}\leq \abs{G}\leq 10(5Y+1)^{d-1}.\]
Since $GG\subseteq B^\times(2Y)$,
\[\Abs{GG}\leq (CY)^{d-1}\leq (C')^d\Delta_K\abs{G}\]
for some absolute constants $C,C'>0$. Let $\epsilon>0$ be some small absolute constant to be chosen soon, and 
\[P=X+B^+(\epsilon X),\]
so that by Lemma~\ref{lem-ball}
\[(\epsilon X)^d\Delta_K^{-1/2}\leq\abs{P}\leq(2\epsilon X+1)^d.\]
Let $A=GP$. We first claim that $\abs{A}=\abs{G}\abs{P}$, for which it suffices to prove that if $u_1/u_2=p_1/p_2$ with $u_i\in G$ and $p_i\in P$ then $u_1= u_2$. This follows since $u_1/u_2\in \mathcal{O}_K^\times$, and for all embeddings $\sigma$ and $p\in P$,
\[\sigma(p)\in [X-\epsilon X, X+\epsilon X],\]
whence 
\[\frac{1-\epsilon}{1+\epsilon}\leq \sigma(p_1/p_2)\leq \frac{1+\epsilon}{1-\epsilon}.\]
Hence, provided $\epsilon>0$ is sufficiently small and $X$ is sufficiently large (which we can assume without loss of generality), $\sigma(p_1/p_2)\in (\phi^{-1},\phi)$. So by Lemma~\ref{lem-unitsep} we have $u_1/u_2\in \{\pm 1\}$, and in fact $u_1/u_2\neq -1$ since otherwise $\sigma(p_1/p_2)=-1$. Therefore there exist constants $0<c<C$ such that
\[\frac{(cXY)^d}{Y\Delta_K^{3/2}}\leq \abs{A}\leq (CXY)^d.\]
For the product set, we note (using the trivial bound $\abs{PP}\leq \abs{P}^2$)
\[\Abs{AA}\leq \Abs{GG}\Abs{PP}\leq  c^{-d}Y^{1-d}\Delta_K^2\abs{A}^2.\]
Finally, every $\alpha\in A$ is an algebraic integer such that $\Abs{\sigma(\alpha)}\leq 2Xe^Y$ for all $\sigma$, and hence $A+A\subseteq B^+(4Xe^Y)$. By Lemma~\ref{lem-ball}
\[\abs{A+A}\leq (Ce^YX)^d\leq (C'e^Y)^d\Delta_K^{1/2}\abs{A}\]
(using $\abs{A}\geq \abs{P}\geq (\epsilon X)^d\Delta_K^{-1/2}$) for some absolute constants $C,C'>0$.
\end{proof}

A similar construction works for the proof of Theorem~\ref{th-mainmult} -- in fact here the construction is even simpler, since we can just take $A=B^\times(Y)$. 
\begin{lemma}\label{lem-mainmult}
There exists an absolute constant $c>0$ such that the following holds. Let $K$ be a totally real number field of degree $d\geq 2$ with discriminant $\Delta_K$ and let $Y\geq 2$. There exists a set $A\subset \mathcal{O}_K$ such that
\[\frac{(cY)^d}{Y\Delta_K^{3/2}}\leq \abs{A}\leq (Y/c)^d\]
and
\[\max(\abs{kA},\Abs{A^{(k)}})\leq (ke^Y/c)^d\textrm{ for any }k\geq 2.\]
\end{lemma}
Once again, Theorem~\ref{th-mainmult} is an immediate consequence, letting $K$ be a totally real number field of large degree $d$ with $\Delta_K\leq C^d$ for some constant $C>0$ and choosing $Y=(C')^{1/\epsilon}$ for some other constant $C'$, so that
\[\max(\abs{kA},\Abs{A^{(k)}})\leq (ke^Y/c)^d\leq \abs{A}^{C^{1/\epsilon}+\epsilon \log k}.\]
This proves the second statement; to prove the first take $\epsilon=C/\log\log k$ for some sufficiently large constant $C>0$ (note that the choice of $A$ then depends on $k$).
\begin{proof}
We argue as in the previous lemma, except that we simply take $A=G=B^\times(Y)$, so that by Lemma~\ref{lem-ball-mult} and Lemma~\ref{lem-regbound}
\[Y^{d-1}d^{-1/2}\Delta_K^{-1}\leq \abs{A}\leq 10(5Y+1)^{d-1}.\]
For any $k\geq 2$, since $A^{(k)}\subseteq B^\times(kY)$,
\[\Abs{A^{(k)}}\leq (CkY)^{d-1}\]
for some absolute constant $C>0$. Furthermore, $A\subseteq B^+(e^Y)$, and hence $kA\subseteq B^+(ke^Y)$, so
\[\abs{kA}\leq (Cke^Y)^d.\qedhere\]
\end{proof}

\section{Numerical estimates}\label{sec-numerical}
We have not tried to keep track of explicit constants in the proofs above, since these would obscure the main ideas of the proof, and the calculations become quite messy. In this section we sketch what a quantified version of the construction would give, yielding in particular $c\geq 0.00000087$. We have not made any attempt to change the structure of the argument, even slightly, to optimize the constants. Doing so would likely yield a better value, although we expect a $c$ obtained by any variant of this kind of argument to be very small.

An earlier version of our argument constructed a field with small split primes, as in the disproof of the unit distance conjecture described in \cite{ABGLSSTWW2026UnitDistance}, and considered elements divisible only by these ideals instead of units. To our surprise, the existence of small split primes turned out to be completely unnecessary, resulting in the simplified version presented here, but it is likely that an optimized version would include these primes.

We first state variants of our lemmas with all the different constants appearing named, and then give explicit values for these constants, before stating a version of the main result for these constants, and then finally giving an explicit value for the main result.

\begin{enumerate}
\item In Lemma~\ref{lem-regbound} we have 
\[R_K \leq c_1^{-d}\Delta_K.\]
\item In Theorem~\ref{th-tower} we have $\Delta_K\leq C_2^d$.
\item In Lemma~\ref{lem-ball}
\[X^d\Delta_K^{-1/2}\leq \Abs{B^+(X)}\leq (2X+1)^d.\]
\item In Lemma~\ref{lem-unitsep} $u\in \{\pm 1\}$ whenever $\Abs{\sigma_i(u)}\in [\tfrac{1}{1+c_3},1+c_3]$.
\item In Lemma~\ref{lem-ball-mult}
\[(1-o(1))^dY^{d-1}R_K^{-1}\leq \Abs{B^\times(Y)}\leq 10(C_4Y+1)^{d-1}.\]
\end{enumerate} 

For example, following the proofs given above we can take 
\[c_1\geq 3.819,\quad C_2\leq 857.57, \quad c_3\geq 0.618,\quad\textrm{and}\quad C_4\leq 4.16.\]
The values of most of the constants here are immediate from the proofs presented above. The constant $C_2$ is sometimes called Martinet's constant (see \cite{HaMa01}). Hajir, Maire, and Ramakrishna \cite[\S3.3.3]{hajir2021shafarevich} proved that we can take $C_2 \leq 857.57$.

For the rest of this sketch we will use the notation $\ls$ and $\gs$ to hide losses of $(1+o(1))^d$ (which are inconsequential since we can take $d$ arbitrarily large). In general, our construction leads to
\[(c_1/C_2)^dY^{d}\ls\abs{G}\leq \abs{GG}\ls (2C_4Y+1)^{d}\]
and, with $\epsilon= \tfrac{c_3}{2+c_3}$ (which is permissible provided $X$ is an integer), 
\[\abs{P}\gs(\epsilon C_2^{-1/2})^dX^d.\]
By discarding elements of $P$ and $G$ if necessary, we can assume that the lower bounds on $\abs{G}$ and $\abs{P}$ are attained, and 
\[(c_1\epsilon C_2^{-3/2})^d(XY)^d\ls \abs{A}.\]
Now 
\[\abs{AA}\ls\brac{\frac{C_2^2(2C_4Y+1)}{c_1^2Y^2}}^d\abs{A}^2\]
and 
\begin{align*}
\abs{A+A}
&\ls (4(1+\epsilon)Xe^Y+1)^d\\
&\ls\brac{\frac{(4(1+\epsilon)Xe^Y+1)C_2^3}{c_1^2\epsilon^2X^2Y^2}}^d\abs{A}^2.
\end{align*}
For example, with the constant choices above we have
\[\abs{AA}\ls\brac{\frac{419531}{Y}+\frac{50425}{Y^2}}^d\abs{A}^2\]
and 
\[\abs{A+A}\ls\brac{\frac{3836812879}{XY^2}e^Y+\frac{776017933}{X^2Y^2}}^d\abs{A}^2,\]
while
\[\abs{A}\gs(0.000035XY)^d.\]
A rough approximation to the optimal choice is to take $X=\lfloor e^{1140402}\rfloor$ and $Y=1140402$, which leads to arbitrarily large $A\subset \bbr$ with
\[\max(\abs{A+A},\abs{AA})\leq \abs{A}^{2-0.00000087}.\]

\section{Linear equations in a multiplicative group}\label{sec-linear}
Let $\Gamma\leq \bbc^\times$ be a multiplicative group. A natural question is how many solutions the equation
\begin{equation}\label{eq-sunit}
x_1+\cdots+x_k=1
\end{equation}
can have with $x_i\in \Gamma$. We are concerned only with non-degenerate solutions, which are those such that $\sum_{i\in I}x_i\neq 0$ for every non-empty $I\subseteq \{1,\ldots,k\}$. Building on a sequence of earlier results, Evertse, Schlickewei, and Schmidt \cite{ESS02} proved the following.
\begin{theorem}[Evertse-Schlickewei-Schmidt]\label{th-ess}
If $\Gamma\leq \bbc^\times$ is a multiplicative group of rank $d$ then, for any $k\geq 2$, the number of non-degenerate solutions to \eqref{eq-sunit} is at most
\[\exp(C_kd)\]
for some constant $C_k>0$ depending only on $k$.
\end{theorem}
They gave $C_k$ as an explicit function of $k$, which has been improved (most recently by Amoroso and Viada \cite{AmVi09}, and is now polynomial in $k$), but here we are most concerned with the dependence on the rank $d$.

Erd\H{o}s, Stewart, and Tijdeman \cite{EST88} constructed,\footnote{In \cite{EST88} they just address the case $k=2$, but this is straightforward to generalise to arbitrary $k\geq 2$ as indicated in \cite{ESS02}.} for any $k\geq 2$ and large enough $d$, multiplicative groups $\Gamma\leq \bbq^\times$ of rank $d$ in which the number of non-degenerate solutions to \eqref{eq-sunit} is at least
\[\exp\brac{ c_k \brac{\frac{d}{\log d}}^{1-\frac{1}{k}}}\]
for some $c_k>0$. This has been improved in some regimes by Konyagin and Soundararajan \cite{KoSo07} (again for $\Gamma\leq \bbq^\times$).

It has been conjectured (see, for example, \cite{ESS02}) that the dependence on $d$ in Theorem~\ref{th-ess} can be improved, perhaps to $\exp(C_kd^{1-c_k})$ for some $c_k>0$. Our construction is also able to disprove this, and shows that the linear dependence on $d$ in the exponent is the best possible. Again, we stress that our construction makes heavy use of algebraic number fields of large degree, and so it remains possible that the dependence on $d$ in the upper bound can be improved if $\Gamma\leq \bbq^\times$, for example.

\begin{theorem}
There is an absolute constant $C>0$ such that the following holds. Let $k\geq C$ be any integer. There exist infinitely many $d\geq 2$ and multiplicative groups $\Gamma\leq\bbr^\times$ of rank $\leq d$ such that there are at least
\[\exp((C^{-1}k\log k)d)\]
many non-degenerate solutions to \eqref{eq-sunit}.
\end{theorem}
\begin{proof}
Let $K$ be a totally real number field of degree $d$ with $\Delta_K\leq C^d$ for some constant $C>0$, as provided by Theorem~\ref{th-tower}. Let $A$ be constructed as in Lemma~\ref{lem-mainmult} (so it is a ball of lattice points in the multiplicative unit lattice of $K$), viewed as a subset of $\bbr$. Note that the unit group has rank $d-1$. Losing only a factor of $2$ in $\abs{A}$ we can assume that $a>0$ for all $a\in A$. We therefore obtain, for any $Y\geq 2$, infinitely many $d$ with accompanying $A$ (contained in a multiplicative group of rank $d-1$) such that $\abs{A}\geq (cY)^{d-1}$ and
\[\abs{kA}\leq (Cke^Y)^d\]
for some constants $c,C>0$. By the pigeonhole principle there exists some $x\in kA$ such that
\[a_1+\cdots+a_k=x\]
has at least
\[\frac{\abs{A}^k}{\abs{kA}}\geq\frac{(cY)^{k(d-1)}}{(Cke^Y)^d}\geq Y^{-k}\left(\frac{(cY)^k}{Cke^Y}\right)^d\]
many solutions. Letting $z_i=a_i/x$, and expanding the group of units with the generator $x$, we achieve at least this many solutions to \eqref{eq-sunit} in a multiplicative group of rank at most $d$. Furthermore, since $z_i>0$, all of these solutions are automatically non-degenerate.

The conclusion then follows from taking $Y=k$.
\end{proof}

A related question is to bound the number of solutions to \eqref{eq-sunit} in the group of units $\mathcal{O}_K^\times$. Evertse \cite{Ev83} proved that the number of solutions to $x_1+x_2=1$ with $x_1,x_2\in \mathcal{O}_K^\times$ is at most $3\cdot 7^{3d}$, where $d$ is the degree of $K$. (Such $x_i$ are often called `exceptional units'.) Niklasch \cite{Ni97} considered this question further and in particular, generalising a conjecture of Stewart (see \cite[p.120]{EGST88}), conjectured \cite[Conjecture 4.2]{Ni97} the sub-exponential upper bound of
\[\exp(d^{2/3+o(1)}).\]
We are able to prove that if we consider the analogous question with $2$ variables replaced by a sufficiently large (but still a constant) number of variables, the analogous conjecture is false, and in fact there are exponentially in $d$ many solutions.
\begin{theorem}\label{th-unitlower-body}
There exists an integer $k\geq 2$ and an absolute constant $C>1$ such that, for infinitely many $d$, there exists a number field $K$ of degree $d$ such that there are at least $C^d$ non-degenerate solutions to \eqref{eq-sunit} with $x_i\in \mathcal{O}_K^\times$.
\end{theorem}
This is, in hindsight, a simple consequence of the fact that units of bounded height still have small height after $O(1)$ many sums, and so the size of their $k$-fold sumset is small. To highlight the simplicity of the example we will present the construction from first principles, at the cost of some slight repetition of earlier arguments.
\begin{proof}
Let $C>0$ be the constant provided by Theorem~\ref{th-tower}, and let $K$ be a number field of degree $d$ and discriminant $\Delta_K\leq C^d$. Let $Y\geq 1$ be some constant to be chosen later, and let $A=B^\times(Y)\subset \mathcal{O}_K^\times$, so that by Lemma~\ref{lem-regbound} and Lemma~\ref{lem-ball-mult}
\[\abs{A}\geq Y^{d-1}d^{-1/2}C^{-d}.\]
Since $B^\times(Y)\subseteq B^+(e^Y)$, $kA\subseteq B^+(ke^Y)$, and hence by Lemma~\ref{lem-ball}
\[\abs{kA}\leq (2ke^Y+1)^d\leq (3ke^Y)^d,\]
say. By the Cauchy-Schwarz inequality it follows that 
\[\#\{x_1+\cdots+x_k=y_1+\cdots+y_k : x_i,y_i\in A\}\geq \frac{\abs{A}^{2k}}{\abs{kA}}\geq \abs{A}^{3k/2},\]
say, provided we first choose $Y$ to be some large constant depending on $C$, and then $k$ some larger constant depending on $Y$. By H\"{o}lder's inequality (see, for example, the proof of \cite[Lemma 5]{ABP25}) the left-hand side is at most
\[C^{k^2}\abs{A}^k+X\]
for some constant $C>1$, where $X$ counts the number of solutions to $x_1+\cdots-y_k=0$ in which no subsum on the left-hand side vanishes. Hence $X\geq \abs{A}^{k/4}$, say, provided $d$ is sufficiently large, which concludes the proof (using that $-1\in \mathcal{O}_K^\times$ and dilating by some fixed $y_k\in A$).
\end{proof}
\section{Variants}\label{sec-variants}

In this section, we discuss three variants of our argument. The first, which requires only minor modifications, disproves the sum-product conjecture in the $p$-adic numbers for each prime $p$. The bounds obtained are uniform in $p$, though we do not make them explicit.

The second, which again requires only minor modifications, produce counterexamples to the strongest form of the sum-product conjecture in all sufficiently large finite fields $\mathbb F_p$ of prime order.

The third, which requires a complete rewrite of the argument, disproves the sum-product conjecture in certain fields of formal Laurent series in characteristic $p$. The bounds on $|A+A|$ and $|AA|$ obtained this way get worse as the characteristic $p$ grows, but for small characteristics, the bounds are much stronger than those obtained from the real version of the argument.

\subsection{The \texorpdfstring{$p$}{p}-adics} We now explain the $p$-adic variant. If the field $K$ in Lemma \ref{lem-main} has a prime lying over $p$ that is split, for example, if $p$ splits completely in $K$, then $K$ embeds into $\mathbb Q_p$ and thus  the set $A \subseteq \mathcal O_K$ constructed in Lemma \ref{lem-main} embeds into $\mathbb Z_p$. Thus, to prove the analogue of Theorem \ref{th-headline} in $\mathbb Z_p$, it suffices to prove the following variant of Theorem \ref{th-tower}.

\begin{lemma}\label{field-existence-split-p}  There exists an absolute constant $C>0$ such that, for every $d$ a power of $2$, there exists a totally real number field $K$ with degree $d$ in which the prime $p$ splits completely such that $\Delta_K\leq C^d$. \end{lemma}

\begin{proof} Let $T = \{p,\infty\}$ and $S=\{3,5,7,11,13,17,19,23\}\backslash \{p\}$. Let $G_S^T(2)$ be the Galois group of the maximal pro-$2$ extension of $\mathbb Q$ unramified outside $S$ and split completely at all primes in $T$. If $G_S^T(2)$ is infinite, then for every $d$ a power of $2$  there is a number fields $K$ with degree $d$ which is totally real (since $\infty$ splits), in which $p$ splits completely, and is ramified only at primes in $S$, with Galois group of order a power of $2$. This is because an infinite pro-$2$-group has open subgroups of index every power of $2$ (because a finite $2$-group has subgroups of index every power of $2$ up to its order).

An extension of fields is called tamely ramified at a prime $p$ if the order of the inertia subgroup at $p$ of the Galois group is coprime to $p$. Since these fields $K$ have Galois group of order a power of $2$, the inertia subgroup has order a power of $2$, so because they are ramified only at odd primes, they are tamely ramified at each ramified prime. It follows by \cite[III, Theorem 2.6]{Ne99} that $\Delta_K \leq (\prod_{q\in S} q )^d \leq C^d$ where $C= 3\cdot 5 \cdot 7 \cdot 11 \cdot 13 \cdot 17 \cdot 19 \cdot 23$. So it only remains to check that $G_S^T(2)$ is infinite.

Let $d(G_S^T(2))$ be the minimum number of generators and $r(G_S^T(2))$ the minimum number of relations in a presentation of $G_S^T(2)$. We can check \[ d(G_S^T(2)) \geq |S|-2\] as follows.  The quadratic field $\mathbb Q( \{ \sqrt{q} \mid q \in S\})$ is split at infinity, ramified only at primes in $S$ and possibly at $2$, and has Galois group $(\mathbb Z/2)^{|S|}$. For $m$ an odd integer, the extension $\mathbb Q_2(\sqrt{m})$ depends only on $m$ mod $8$. It follows that \[\mathbb Q_2( \{ \sqrt{q} \mid q \in S\})= \mathbb Q_2( 1,\sqrt{-1}, \sqrt{5},\sqrt{-5})=\mathbb Q_2( \sqrt{-1},\sqrt{5})\] where we have chosen one representative from each congruence class mod $8$. Since $\mathbb Q_2(\sqrt{5})$ is unramified, $\mathbb Q_2( \sqrt{-1},\sqrt{5})$ has inertia group of order $2$. Thus the inertia group at $2$ is a subgroup of order $2$ of the Galois group. Taking the quotient by this we get a field with Galois group $(\mathbb Z/2)^{|S|-1}$ that is ramified only at primes in $S$. The Frobenius element at $p$ is an element of this Galois group, and thus has order $1$ or $2$. Taking the quotient by this element, we get a field with Galois group $(\mathbb Z/2)^{|S|-2}$ or $(\mathbb Z/2)^{|S|-1}$ which in addition splits completely at $p$. Thus $(\mathbb Z/2)^{|S|-2}$ is a quotient of $G_S^T(2)$ and hence 
\begin{equation}\label{Zp-case-1} d(G_S^T(2)) \geq |S|-2\geq 5.\end{equation}

We have
\begin{equation}\label{Zp-case-2} r(G_S^T(2)) \leq  d(G^T_S(2))+1 \end{equation}
by \cite[Theorem 10.7.12]{Neukirch2008}, since, in the notation of \cite{Neukirch2008}, $\chi(G_S^T(2))=1+r(G_S^T(2))-d(G_S^T(2))$ and we have $\theta=0$, $S$ does not intersect $S_p =\{p\}$, $r=1$, and $T \setminus S_\infty= \{p\}$ has cardinality $1$, so that $\chi(G_S^T(2))\leq 0+0+1+1=2$ and hence $r(G_S^T(2))-d(G_S^T(2))\leq 1$.

It follows from \eqref{Zp-case-1} and \eqref{Zp-case-2} that
\[ r(G_S^T(2)) <  \frac{d(G^T_S(2))^2}{4} \]
and hence by the Golod-Shafarevich theorem~\cite{GolodShafarevich} in its refined form due to Gasch\"utz and Vinberg~\cite{Vinberg1965,Koch1969}, $G_S^T(2)$ is infinite, as desired.\end{proof}

\subsection{Finite fields}

\begin{theorem}\label{finite-field-main} There exist constants $c>0$ and $f<1$ such that for each $\delta \in (0,1)$, for each prime $p$ sufficiently large depending on $\delta$, there exists $A \subset \mathbb F_p$ with $p^{f\delta}< \abs{A} < p^{\delta}$ and $\max(\abs{A+A},\abs{AA})\leq \abs{A}^{2-c}$. \end{theorem}

Sum-product results in finite fields typically require both an upper bound and a lower bound on the size of $A$, so we have stated this result with an upper bound and a lower bound.

It may be possible to prove a result with tighter control on $\abs{A}$ in terms of $p$, by choosing $p$ after constructing a field $K$ and a subset of $\mathcal O_K$, at the cost that this result would hold for infinitely many primes instead of all primes.

The value of $c$ obtained from our argument is very small. It is slightly worse than the explicit value of $c$ we obtain for the main theorem, owing to the dependence of this argument on Lemma \ref{field-existence-split-p}. It may be possible to prove a similar result in finite fields of large size and small characteristic, with a better exponent, using the results of the next subsection and reducing modulo a prime of the function field $\mathbb F_q(C)$. To make this interesting, one would have to check that the sets produced this way are far from any subfield of the finite field, for example as in the finite field sum-product estimate of Li and Roche-Newton \cite{Li2011}.

\begin{proof} We apply Lemma~\ref{field-existence-split-p} for a $d$ to be chosen later to produce a number field $K$ of degree $d$ in which $p$ splits completely. We apply Lemma~\ref{lem-main} to produce $A \subset \mathcal O_K$. Since the prime $p$ splits completely in $K$, we may choose a prime $\mathfrak p$ of $\mathcal O_K$ lying over $p$, with residue field $\mathbb F_p$, to obtain a surjection $\mathcal O_K \to \mathbb F_p$. We will consider the image of $A$ inside $\mathbb F_p$.

In the proof of Lemma~\ref{lem-main}, it is observed that every $\alpha \in A$ has $\abs{\sigma(\alpha)}\leq 2X e^Y$ for all $\sigma$, and thus for $\alpha_1 \neq \alpha_2$ in $A$ we have $\abs{\sigma(\alpha_1-\alpha_2)}\leq 4 X e^Y$ and hence the norm of $\alpha_1-\alpha_2$, which is the product of its image under all embeddings $\sigma$, is at most $(4X e^Y)^d$. If $\alpha_1 $ and $\alpha_2$ have the same image in $\mathbb F_p$, then $\alpha_1-\alpha_2$ must be divisible by $\mathfrak p$ and hence have a norm a multiple of $p$.

It follows that for the map $\mathcal O_K \to\mathbb F_p$ to be injective on $A$, it suffices to have $(4 Xe^Y)^d <p$.

From Lemma \ref{lem-main}, $\abs{A} \leq (XY/c')^d$ for an absolute constant $c'$. Thus to have $\abs{A}<p^\delta$, it suffices to have $(XY/c')^{d}<p^\delta$. Let $d$ be the least power of $2$ such that $(4 Xe^Y)^d <p$ and $(XY/c')^{d}<p^\delta$. Arguing as in the proof of Theorems \ref{th-main} and \ref{th-headline}, we have $\max(\abs{A+A},\abs{AA})\leq \abs{A}^{2-c}$. Since $p$ is sufficiently large, $d$ is sufficiently large to be used in this argument.

It remains to prove $\abs{A}> p^{\frac{\delta}{f}}$. To do this, we use Lemma \ref{lem-main} which gives
\begin{equation}\label{A-lb-fin-case} \abs{A} \geq \frac{ (c' XY)^d}{ Y \Delta_K^{3/2}} \geq  \frac{ (c' XY)^d}{ Y C^{3d/2}}\end{equation} which, with parameters chosen as in the proof of Theorems \ref{th-main} and \ref{th-headline}, is exponentially large in $d$. Since $d$ is the least power of $2$ such that $(4 Xe^Y)^d <p$ and $(XY/c')^{\frac{d}{\delta}}<p$, we have either $(4 Xe^Y)^{2d} \geq p$ or $(XY/c')^{\frac{2d}{\delta}}\geq p$. Combining either one of these with \eqref{A-lb-fin-case} gives a lower bound of a power of $p^\delta$, as desired.\end{proof}

\subsection{Function fields}

We now construct counterexamples to the sum-product conjecture in fields of characteristic $p$. The constructions will lie in a sequence of fields $\mathbb F_q(C_i)$ for a sequence of algebraic curves $C_i$, and hence give counterexamples to the sum-product conjecture in any field containing all of them as subfields, such as $\overline{\mathbb F_q(t)}$ or $\mathbb F_q((t))$ (since our curves $C_i$ will have rational points so that $\mathbb F_q(C_i) \subseteq \mathbb F_q((t))$).

The rational places of $C_i$ will play the role that the infinite places play in the main argument of this paper, or that the small split primes play in the original unit distance argument. Hence we rely on constructions of curves with many rational points.

Let $C$ be a smooth projective geometrically connected curve over a finite field $\mathbb F_q$, and let $\mathbb F_q(C)$ be a field of rational functions on $C$. A convenient way to produce a subset $A \subset \mathbb F_q(C)$ is to construct a subset $A \subseteq H^0(C,L)$ of the global sections $H^0(C,L)$ of a line bundle $L$ on $C$. Dividing by any nonzero section of $L$ identifies $A \subset H^0(C,L)$ with a subset of $\mathbb F_q(C)$. This operation is compatible with taking sums and products, so to find a counterexample to sum-product it suffices to find a subset $A \subseteq H^0(C,L)$ such that $A + A \subseteq H^0(C,L)$ and $AA \subseteq H^0(C,L^{ 2})$ are both small. 

Our construction is as follows. Let $L_P$ be a line bundle of degree $d_P$ and $L_G$ be a line bundle of degree $d_G$. Let 
\[ P = \{ f \in H^0(C, L_P) \mid f \textrm{ does not vanish at any point in } C(\mathbb F_q)\}\]
and
\[G = \{ g \in H^0(C, L_G) \mid g \textrm{ vanishes only at points in } C(\mathbb F_q)\}.\]
Since the $0$ section vanishes everywhere, $0$ is contained in neither $P$ nor $G$.

Let $A = P G \subseteq H^0(C, L_P \otimes L_G)$. 

To understand the analogy between this construction and our original construction with number fields, one should think of $C(\mathbb F_q)$ as analogous to the set of infinite places and $H^0(C,L)$ as analogous to the set of elements of the ring of integers with bounded absolute value at each infinite place, with the exact bound depending on the line bundle $L$. Then $P$ is analogous to the set of elements of the ring with bounded absolute value at each infinite place, that are also not too small at each infinite place (since the nonvanishing at $x\in C(\mathbb F_q)$ forces the $x$-adic absolute value to not be too large), which is exactly how the set $P$ in the number field case can be described. The elements of $G$ are analogous to the set of elements of the ring of integers with bounded absolute value at each infinite place that are not divisible by any finite prime (since vanishing at a point is equivalent to being divisible by the corresponding prime ideal), in other words, units of bounded absolute value, which is similar to the construction of $G$ in the number field case. (We have dropped the lower bound on the absolute value that was used in the number field case.)

\begin{lemma}\label{ff-pg} We have $ |PG| = \frac{|P| |G|}{q-1}.$\end{lemma}

\begin{proof}  If $f_1,f_2 \in P$ and $g_1,g_2\in G$ satisfy $f_1g_1=f_2g_2$ then we have $f_1/f_2=g_2/g_1$. Since $f_1/f_2$ is a rational function with no zeroes or poles at points of $C(\mathbb F_q)$, and $g_2/g_1$ is a rational function with only zeroes and poles at points of $C(\mathbb F_q)$, they must both have no zeroes or poles and hence be elements of $\mathbb F_q^\times$. Thus $PG = (P \times G)/\mathbb F_q^\times$.\end{proof}

Let $g$ be the genus of $C$.

\begin{lemma}\label{ff-rr} As long as $d_P \geq 2g-1 +\abs{C(\mathbb F_q)}$ we have 
\begin{equation}\label{ff-P-bound} |P| = q^{ d_P + 1-g} (1-q^{-1})^{|C(\mathbb F_q)|} \end{equation}
and
\begin{equation}\label{ff-sum-bound}|PG + PG| \leq q^{ d_P+d_G  + 1-g} .\end{equation}
\end{lemma}
\begin{proof} These follow by the Riemann-Roch formula.

For \eqref{ff-sum-bound}, note that $PG+ PG$ is a subset of $H^0(C, L_P \otimes L_G)$. If $G$ is non-empty we have $d_G\geq 0$ so the assumption implies $d_P +d_G \geq   2g-1 +\abs{ C(\mathbb F_q)}\geq 2g-1$ and thus 
\[|PG + PG|  \leq |H^0(C,L_P \otimes L_G)| =q^{ \deg( L_P \otimes L_G )+ 1-g} = q^{d_P +d_G +1-g} \] by Riemann-Roch.

For \eqref{ff-P-bound}, we use inclusion-exclusion to obtain
\[ |P|= \sum_{S \subseteq C(\mathbb F_q)} (-1)^{|S|} |\{ f\in H^0(C, L_P) | f \textrm{ vanishes at all points in } S\}|\]\[=  \sum_{S \subseteq C(\mathbb F_q)}  (-1)^{|S|} |H^0(C, L_P (- \sum_{x\in S}[x]))| =  \sum_{S \subseteq C(\mathbb F_q)}  (-1)^{|S|}  q^{ \deg L_P(- \sum_{x\in S}[x])+1-g} \]\[= \sum_{S \subseteq C(\mathbb F_q)}  (-1)^{|S|} q^{ d_P - |S| + 1-g}= q^{ d_P + 1-g} (1-q^{-1})^{|C(\mathbb F_q)|} \]
since $\deg L_P(- \sum_{x\in S}[x]) = d_P - |S| \geq d_P - |C(\mathbb F_q)|  \geq 2g-1$ by assumption. \end{proof}

Let $\operatorname{Pic}^0(C)(\mathbb F_q)$ be the degree-zero Picard group of $C$ (which goes by other names, including the $\mathbb F_q$-points of the Jacobian of $C$ and the class group of $\mathbb F_q(C)$) and let $\operatorname{Pic}^0(C)(\mathbb F_q)[2]$ be its $2$-torsion subgroup. For $L$ a line bundle, let 
\[ N_{\mathbb F_q}(L) = |\{ g \in H^0(C, L) \mid g \textrm{ vanishes only at points in } C(\mathbb F_q)\}|.\]

\begin{lemma}\label{ff-pm} For each $d_G \geq 0$ there exists a line bundle $L_G$ of degree $d_G$ such that
\begin{equation}\label{ff-G-bound} N_{\mathbb F_q}(L_G) \geq \frac{ (q-1)   \binom{d_G + |C(\mathbb F_q)|-1}{d_G}}{ 2|\operatorname{Pic}^0(C)(\mathbb F_q)|}\end{equation}
and
\begin{equation}\label{ff-almost-product-bound} \frac{ N_{\mathbb F_q}(L_G^2)}{N_{\mathbb F_q}(L_G)} \leq \frac{ 2 |\operatorname{Pic}^0(C)(\mathbb F_q)[2]| \binom{2d_G + |C(\mathbb F_q)|-1}{2d_G}}{ \binom{d_G + |C(\mathbb F_q)|-1}{d_G}}.\end{equation}\end{lemma}

\begin{proof} If we choose $L_G$ uniformly at random among isomorphism classes of line bundles $L_G$ of degree $d_G$, letting $\mathbb E$ be the expectation, we have \[ \mathbb E[ N_{\mathbb F_q}(L_G)] =\frac{ (q-1)   \binom{d_G + |C(\mathbb F_q)|-1}{d_G}}{ |\operatorname{Pic}^0(C)(\mathbb F_q)|}\] since there are $ \binom{d_G + |C(\mathbb F_q)|-1}{d_G}$ divisors of degree $d_G$ supported at the points of $C(\mathbb F_q)$, each divisor defines $q-1$ sections of $L_G$ if its divisor class equals the divisor class of $L_G$ and $0$ sections otherwise, and the number of divisor classes of degree $d_G$ is $|\operatorname{Pic}^0(C)(\mathbb F_q)|$.

If we choose $L_G$ uniformly at random, then the divisor class of $L_G^{2}$ is chosen uniformly at random from the divisor classes of degree $2d_G$ that are divisible by $2$, the number of which is $\frac{| \operatorname{Pic}^0(C)(\mathbb F_q)|}{|\operatorname{Pic}^0(C)(\mathbb F_q)[2]|}$. Thus \[ \hspace{-.5in} \mathbb E[ N_{\mathbb F_q}(L_G^2) ] \leq \frac{ (q-1)   \binom{2d_G + |C(\mathbb F_q)|-1}{2d_G} |\operatorname{Pic}^0(C)(\mathbb F_q)[2]| }{ |\operatorname{Pic}^0(C)(\mathbb F_q)|}\] by the same reasoning. Hence we can choose $L_G$ of degree $d_G$ with 
\begin{equation}\label{ff-expectation-inequality} \frac{ |\operatorname{Pic}^0(C)(\mathbb F_q)|}{ (q-1)   \binom{d_G + |C(\mathbb F_q)|-1}{d_G}}N_{\mathbb F_q}(L_G) - \frac{ |\operatorname{Pic}^0(C)(\mathbb F_q)|}{2 (q-1)   \binom{2d_G + |C(\mathbb F_q)|-1}{2d_G} |\operatorname{Pic}^0(C)(\mathbb F_q)[2]| }N_{\mathbb F_q}(L_G^2) \geq \frac{1}{2} \end{equation}
as the expectation of the left hand side of \eqref{ff-expectation-inequality} is at least $\frac{1}{2}$ when $L_G$ is chosen uniformly at random and thus we can choose an $L_G$ where the left hand side is at least $\frac{1}{2}$.

\eqref{ff-expectation-inequality} immediately implies \eqref{ff-G-bound} by dropping the $N_{\mathbb F_q}(L_G^2)$ term and implies \eqref{ff-almost-product-bound} by dropping the $\frac{1}{2}$ term.\end{proof}

\begin{lemma} \label{ff-product} Choosing $L_G$ as in Lemma \ref{ff-pm}, we have
\begin{equation}\label{ff-G-product-bound} \frac{ |GG|}{|G|} \leq \frac{ 2 |\operatorname{Pic}^0(C)(\mathbb F_q)[2]| \binom{2d_G + |C(\mathbb F_q)|-1}{2d_G}}{ \binom{d_G + |C(\mathbb F_q)|-1}{d_G}}\end{equation}
and as long as $d_P \geq 2g-1 +\# C(\mathbb F_q)$ we have
\begin{equation}\label{ff-product-bound} \frac{|PGPG|}{|PG|}\leq \frac{ 2 |\operatorname{Pic}^0(C)(\mathbb F_q)[2]| \binom{2d_G + |C(\mathbb F_q)|-1}{2d_G} q^{ d_P + 1-g} (1-q^{-1})^{|C(\mathbb F_q)|}}{ \binom{d_G + |C(\mathbb F_q)|-1}{d_G}}.\end{equation} 
\end{lemma}

\begin{proof} We have \[ \{g \in H^0(C, L_G) \mid g \textrm{ vanishes only at points in } C(\mathbb F_q)\}^2\]\[ \subseteq \{ h \in H^0(C, L_G^2) \mid h \textrm{ vanishes only at points in } C(\mathbb F_q)\} \] which together with \eqref{ff-almost-product-bound} implies \eqref{ff-G-product-bound}.

We have $ |PGPG| \leq |PP| |GG|/(q-1) \leq |P|^2 |GG|/(q-1)$ since both $PP$ and $GG$ are stable under multiplication by $\mathbb F_q^\times$ and we have $|PG|=|P||G|/(q-1)$ by Lemma \ref{ff-pg} so we have
\[ \frac{ |PGPG|}{|PG|} \leq \frac{ |P| |GG|}{|G|} \] so that \eqref{ff-product-bound} follows from \eqref{ff-G-product-bound} and \eqref{ff-P-bound}.
\end{proof}

Putting this together, we set $A = PG$ and then embed $A$ into $\mathbb F_q(C)$. Using  Lemma \ref{ff-pg}, \eqref{ff-P-bound}, and \eqref{ff-G-bound}, we obtain

 \begin{equation}\label{ff-a} |A| \geq \frac{ \binom{d_G + |C(\mathbb F_q)|-1}{d_G} q^{ d_P + 1-g} (1-q^{-1})^{|C(\mathbb F_q)|}}{ 2|\operatorname{Pic}^0(C)(\mathbb F_q)|}.\end{equation} Using \eqref{ff-sum-bound}, we obtain  \begin{equation}\label{ff-a-sum} | A + A |\leq q^{ d_P+d_G  + 1-g}.\end{equation} Using \eqref{ff-product-bound}, we obtain \begin{equation}\label{ff-a-product}\frac{|AA|}{|A|} \leq \frac{ 2 |\operatorname{Pic}^0(C)(\mathbb F_q)[2]| \binom{2d_G + |C(\mathbb F_q)|-1}{2d_G} q^{ d_P + 1-g} (1-q^{-1})^{|C(\mathbb F_q)|}}{ \binom{d_G + |C(\mathbb F_q)|-1}{d_G}}.\end{equation}

We first give a counterexample to the sum-product theorem in $\mathbb F_p((t))$, though with the exponent getting worse as the characteristic grows. This argument is not particularly optimized. Afterwards, we give an argument that gets a more optimized exponent in $\mathbb F_q((t)))$ for specific finite fields $\mathbb F_q$. This second result is specialized to the case of $q$ a perfect square, to take advantage of known constructions of curves with many rational points over finite fields of square order.

We begin with an asymptotic formula for binomial coefficients For $x,y$ positive reals, let $F_q(x,y) = (x+y)\log_q(x+y) - x \log_q(x) - y\log_q(y)$.

We have the asymptotic
\begin{equation}\label{binomial-asymptotic} \binom{n+m}{n} = q^{ F_q (x,y) g + o(g)} \textrm{ when} n =xg +o(g) \textrm{ and } m=yg + o(g) \end{equation} that follows from Stirling's formula.

\begin{theorem}\label{th-ffspec} There is an absolute constant $c>0$ such that for any prime $p$, there exist finite subsets $A \subset \mathbb F_{p}((t))$ of arbitrarily large cardinality such that $|A + A| \leq |A|^{2 -\frac{c}{\log p}} $ and $|AA| \leq |A|^{2-\frac{c}{\log p}}$.\end{theorem}

Theorem~\ref{th-ffmain} immediately follows, since for $q$ a power of $p$, $\mathbb F_{q}((t))$ contains $\mathbb F_{p}((t))$ and so the $p$ case implies the general case.

\begin{proof} It was proven by Serre~\cite{Serre83} (but see \cite[Appendix]{Elkies2004} for the proof) that there exists an absolute constant $d>0$ such that for each finite field $\mathbb F_q$ there exists $C$ over $\mathbb F_p$ with genus $g(C)$ arbitrarily large such that \begin{equation}\label{points-on-curve-serre} \abs{C(\mathbb F_q)}  \geq d g(C) \log(q) \end{equation}

We have
\begin{equation}\label{jacobian-trivial} \abs{\operatorname{Pic}^0(C)(\mathbb F_q)} \leq (\sqrt{q}+1)^{2g} \end{equation}
by Weil's Riemann hypothesis for curves.

 Finally, we have the bound
 \begin{equation}\label{2-torsion-trivial} \abs{\operatorname{Pic}^0(C)[2](\mathbb F_q)}\leq 2^{2 g(C)}\end{equation} valid since $\operatorname{Pic}^0(C)$ is an abelian variety of dimension $g(C)$ and thus has at most $2^{2g(C)}$ two-torsion points. 

We take $q=p$ and take $d_P = d_G = x  \abs{ C(\mathbb F_p)}$ for some absolute but sufficiently large integer $x$.  We have 
\[d_P = x \abs{C(\mathbb F_q)} \geq xd g(C) \log (q) \geq xd g(C) \log 2 \geq 2g \]
for $x$ sufficiently large. 

From \eqref{ff-a}, \eqref{binomial-asymptotic}, \eqref{jacobian-trivial}, and \eqref{points-on-curve-serre} we get
\[ \log_q \abs{A}= \abs{C(\mathbb F_q)} (F_q(x,1) + x + \log_q(1-q^{-1}) +o(1))  - g(C) ( 1+ 2\log_q (\sqrt{q}+1)) \]
\[ =\abs{C(\mathbb F_q)} \left(F_q(x,1) + x + O\left( \frac{1}{\log q} \right) + o(1) \right) \]
since $2\log_q (\sqrt{q}+1)=O(1)$ and $g(c) = O ( \frac{1}{\log q} ) \abs{C(\mathbb F_q)} $  and $\log_q(1-q^{-1})= O ( \frac{1}{\log q} ) $ also.

By  \eqref{ff-a-sum}, we get
\[ \log_q \abs{A+A} \leq  2x \abs{C(\mathbb F_q)} .\]

Using \eqref{ff-a-product}, \eqref{2-torsion-trivial}, \eqref{binomial-asymptotic}, and \eqref{points-on-curve-serre}, we get
\[ \log_q \left(\frac{\abs{AA}}{\abs{A}}\right) \leq  \abs{C(\mathbb F_q)} ( F_q(2x,1) - F_q(x,1) +x+ o(1) )  + g ( 2 \log_q(2) )\]
\[ = \abs{C(\mathbb F_q)} ( F_q(2x,1) - F_q(x,1) +x+O\left(\frac{1}{\log q}\right) +o(1) ).\]

Now \[F_q(x,1) =\frac{ ( (x+1) \log(x+1)-x\log x)}{\log q} = {\log(x+1) + x \log(1+1/x)} {\log q}.\] We can choose $x$ sufficiently large that $F_q(x,1)$ is greater than the $O(\frac{1}{\log q})$ term by some positive multiple of $\frac{1}{\log q}$, in which case  $ \frac{\log_q \abs{A+A}}{\log_q \abs{A}} $ will be at most $2-c/\log q$ for some $c>0$, as desired. We have
\[ \log_q\left( \frac{ \abs{A}^2}{\abs{AA}} \right) \geq \abs{C(\mathbb F_q)} ( 2F_q(x,1) - F_q(2x,1) +O \left(\frac{1}{\log q}+o(1)\right)\]
and
\[  2F_q(x,1) - F_q(2x,1) ={ 2 \log x -\log(x+1) + 2x \log(1+1/x)- 2x \log(1+1/(2x))}{\log q}.\]
We can choose $x$ sufficiently large that $ 2F_q(x,1) - F_q(2x,1) $ is greater than the $O(\frac{1}{\log q})$ term by some positive multiple of $\frac{1}{\log q}$, in which case we have
\[ \frac{ \log_q( \frac{ \abs{A}^2}{\abs{AA}} )}{\log_q(\abs{A})}  \geq \frac{c}{\log q} \] for some $c>0$, as desired, since the denominator is $O( \abs{C(\mathbb F_q)})$. \end{proof}

Finally, we give a more optimized version of the proof of Theorem~\ref{th-ffspec} over fields of perfect square order. As promised, this argument delivers exponents close to $1.9$ for some values of $q$.
 
 \begin{theorem}\label{ff-complicated} Let $q$ be a prime power that is a perfect square. Let $a,b \in (1,2)$. Then there exist finite subsets $A \subset \mathbb F_{q}((t))$ of arbitrarily large cardinality such that $|A + A| \leq |A|^a$ and $|AA| \leq |A|^b$ as long as there exist $\beta>0$ and $\alpha> \sqrt{q}+1$ such that
 \begin{equation}\label{ff-a-condition} a >\frac{ \alpha + \beta - 1}{ F_q(\beta, \sqrt{q}-1) + \alpha + 2(\sqrt{q}-1) \log_q (1- q^{-1}) -2 } \end{equation}
 \begin{equation}\label{ff-b-condition} b > 1 + \frac{2 \log_q( 2) + F_q(2\beta, \sqrt{q}-1) - F_q(\beta, \sqrt{q}-1) + \alpha -1+ ( \sqrt{q}-1) \log_q(1-q^{-1} )  }{ F_q(\beta, \sqrt{q}-1) + \alpha +  2(\sqrt{q}-1) \log_q (1- q^{-1}) -2 }. \end{equation}
 
 If $q$ is a power of $2$, we may replace \eqref{ff-b-condition} by the weaker 
 \begin{equation}\label{ff-b-condition-char2}  b > 1 + \frac{\frac{ \log_q( 2) }{\sqrt{q}+1} + F_q(2\beta, \sqrt{q}-1) - F_q(\beta, \sqrt{q}-1) + \alpha + ( \sqrt{q}-1) \log_q(1-q^{-1} )  -1}{ F_q(\beta, \sqrt{q}-1) + \alpha +  2(\sqrt{q}-1) \log_q (1- q^{-1}) -2}.\end{equation} \end{theorem}
 
 \begin{proof} Note that under the assumptions, the denominator $F_q(\beta, \sqrt{q}-1) + \alpha + 2(\sqrt{q}-1) \log_q (1- q^{-1}) -2 $ is always at least $\sqrt{q}-1 + 2(\sqrt{q}-1) \log_q(1-q^{-1})$ and thus always positive since $\log_q(1-q^{-1}) \geq \log_4(3/4) > -\frac{1}{2}$ as $q\geq 4$.
 
 We choose a sequence of curves $C_i$ with $g_i$ tending to $\infty$ and \begin{equation}\label{points-on-curve-limit}\lim_{i \to \infty} \frac{ |C_i(\mathbb F_q)|}{g(C_i)} = \sqrt{q}-1.\end{equation} That such a sequence exists for $q$ a perfect square was proven independently by Ihara~\cite{Ih82} and by Tsfasman, Vl\u{a}du\c{t}, and Zink~\cite{Tsfasman1982}. That this is optimal was proven by Drinfeld and Vl\u{a}du\c{t}~\cite{Vladut1983}.  For such a sequence, the limit
 \begin{equation}\label{jacobian-limit} \lim_{i \to \infty} \frac{\log_q  |\operatorname{Pic}^0(C_i)(\mathbb F_q)|}{g(C_i)} = 1- ( \sqrt{q}-1) \log_q (1- q^{-1})\end{equation}
 was established by Rosenbloom and Tsfasman~\cite[Lemma A.2]{Rosenbloom1990}.

 We choose $d_{P_i} = \lceil \alpha g(C_i) \rceil$ and $d_{G_i} = \lceil \beta g(C_i) \rceil$ and construct a set $A_i =P_iG_i$ as described above. We have $A_i \subseteq \mathbb F_q(C_i)$. For $i$ sufficiently large, we can embed $\mathbb F_q(C_i)$ into $\mathbb F_q((t))$ using any rational point of $C_i$, of which there are many by \eqref{points-on-curve-limit}, so $A_i$ will indeed define a subset of $\mathbb F_q((t))$.
 
 We have $d_{P_i} \geq 2g(C_I)-1+ |C_i(\mathbb F_q)| $ for $i$ sufficiently large since  $\alpha> \sqrt{q}+1$.
 
 From \eqref{ff-a}, \eqref{binomial-asymptotic}, \eqref{points-on-curve-limit}, and \eqref{jacobian-limit}, we have that
\[ \frac{ \log_q( |A_i|)}{g(C_i)} \geq  F_q(\beta, \sqrt{q}-1) + \alpha -1 + (\sqrt{q}-1) \log_q (1- q^{-1})  - 1 + (\sqrt{q}-1) \log_q(1-q^{-1})+o(1) \] \[= F_q(\beta, \sqrt{q}-1) + \alpha +2(\sqrt{q}-1) \log_q (1- q^{-1}) -2+o(1) .  \]

From \eqref{ff-a-sum} we see that
\[ \frac{\log_q( |A_i+A_i|)}{g(C_i)} \leq \alpha + \beta -1+o(1). \]
From \eqref{ff-a-product}, \eqref{2-torsion-trivial}, \eqref{binomial-asymptotic}, and \eqref{points-on-curve-limit}, we see that
\[ \frac{\log_q( |A_iA_i|) -\log_q(|A_i|)}{g(C_i)} \] \[\leq 2 \log_q(2) + F_q(2\beta, \sqrt{q}-1) + \alpha -1 +  (\sqrt{q}-1) \log_q(1-q^{-1})  - F_q(\beta, \sqrt{q}-1)+o(1).\]

From these and \eqref{ff-a-condition} it follows that $|A_i+A_i|\leq |A_i|^a$ for $i$ sufficiently large, and from these and \eqref{ff-b-condition} it follows that $|A_i A_i| \leq |A_i|^b$ for $i$ sufficiently large.

Finally, in the case when $q$ is a power of $2$,  there exists a sequence $C_i$ satisfying \eqref{points-on-curve-limit} and thus \eqref{jacobian-limit} but also
\begin{equation}\label{2-torsion-char2-bound} \lim_{i \to \infty} \frac{\log_q  |\operatorname{Pic}^0(C_i)[2](\mathbb F_q)|}{g(C_i)} = \frac{\log_q(2)}{\sqrt{q}+1} \end{equation} was proven by Cascudo, Cramer, and Xing~\cite[Theorem 2.3(iii)]{CCX}. Replacing \eqref{2-torsion-trivial} with \eqref{2-torsion-char2-bound} in the above argument, we obtain the same conclusion under \eqref{ff-b-condition-char2}. \end{proof}

 In small characteristic, one can obtain explicit exponents close to $1.9$, as long as the finite field size is large enough. For example, if $q=1024$ we can take $a=b=1.906$ since we may take $\alpha=33.01$ and $\beta=40.53$. For $q=41^2$ we can take $a=1.910$ and $b=1.912$ since we may take $\alpha=42.01$ and $b=51.5$.
 
 Over very small finite fields, the exponents are slightly worse. For example, if $q=4$ we can take  $ a=1.939$ and $b=1.941$ since we may take $\alpha=10.75$ and $\beta=11.25$. If $q=9$ we can take $a=1.964$ and $b=1.972$ since we may take $\alpha=11.5$ and $\beta=13$.

\bibliographystyle{plain}
\bibliography{SPCrefs}

\end{document}